\documentclass[reqno,10pt]{amsart}

\usepackage{graphicx}
\usepackage{comment}
\usepackage{mathdots}
\usepackage{amssymb}
\usepackage{amsmath}
\usepackage{amsfonts}
\usepackage{enumerate}
\usepackage{mathrsfs}
\usepackage{color}
\usepackage{kbordermatrix}
\usepackage[normalem]{ulem}
\usepackage{mathdots}%
\usepackage{url}%

\newcommand{\mcal}{\mathcal}
\newcommand{\q}{\left\{}
\newcommand{\w}{\right\}}
\newcommand{\re}{\mathbb{R}}

\newcommand{\tb}{\textbf}

\newcommand{\ti}{\textit}

\newcommand{\bt}{\beta}

\newcommand{\sbs}{\subseteq}

\newcommand{\opn}{\operatorname}
\newcommand{\rank}{\opn{rank}\,}

\newcommand{\supp}{\opn{supp}\,}

\newcommand{\lp}{\left(}
\newcommand{\rp}{\right)}

\newcommand{\V}{\mathcal {V}}

\renewcommand{\kbldelim}{(}
\renewcommand{\kbrdelim}{)}
\renewcommand{\kbrowstyle}{\displaystyle}
\renewcommand{\kbcolstyle}{\displaystyle}

\newcommand{\boldt}{\mathbf{t}}
\newcommand{\boldw}{\mathbf{w}}
\newcommand{\RR}{\mathbb R}
\newcommand{\NN}{\mathbb N}
\newcommand{\ZZ}{\mathbb Z}
\newcommand{\calZ}{\mathcal Z}
\newcommand{\cZ}{\mathcal Z}

\newcommand{\cI}{\mathcal I}

\newcommand{\ut}{\underline{t}}

\newcommand{\mbf}{\mathbf}

\theoremstyle{plain}
\newtheorem{theorem}{Theorem}[section]
\newtheorem{proposition}[theorem]{Proposition}

\theoremstyle{definition}

\newtheorem{example}[theorem]{Example}

\theoremstyle{remark}
\newtheorem{remark}[theorem]{Remark}

\begin{document}

\title[TMP on the curve $xy=x^m+q(x)$]{Bivariate Truncated Moment Sequences with the Column Relation $XY=X^m + q(X)$, with $q$ of degree $m-1$}

\author[S. Yoo]{ Seonguk Yoo}
\address{Seonguk Yoo \\  Department of Mathematics Education and RINS \\ Gyeongsang National University \\ Jinju, Republic of Korea 52828}
\email{seyoo@gnu.ac.kr}

\author[A.\ Zalar]{Alja\v z Zalar}
\address{Alja\v z Zalar\\ Faculty of Computer and Information Science,
University of Ljubljana \&
Faculty of Mathematics and Physics, University of Ljubljana  \&
Institute of Mathematics, Physics and Mechanics, Ljubljana, Slovenia.}
\email{aljaz.zalar@fri.uni-lj.si}

\thanks{The second-named author was supported by the ARIS (Slovenian Research and Innovation Agency)
research core funding No.\ P1-0228 and grant No.\ J1-50002.}



\begin{abstract} 
When the algebraic variety associated with a truncated moment sequence is finite, solving the moment problem follows a well-defined procedure. However, moment problems involving infinite algebraic varieties are more complex and less well-understood. 
Recent studies suggest that certain bivariate moment sequences can be transformed into equivalent univariate sequences, offering a valuable approach for solving these problems. In this paper, we focus on addressing the truncated moment problem (TMP) for specific rational plane curves. 
For a curve of general degree we derive an equivalent Hankel positive semidefinite completion problem. 
For cubic curves, we solve this problem explicitly, which resolves the TMP for one of the four types of cubic curves, up to affine linear equivalence. 
For the quartic case we simplify the completion problem to a feasibility question of a three-variable system of inequalities.
\end{abstract}
\maketitle

%\section{Bivariate Truncated Moment Problems }
\section{Introduction}

%%%%%%% intro of TMP
Given a real 2-dimensional multisequence of  degree $m$,  $\beta \equiv\beta^{(m)}= \{  \beta_{00},$ $\beta_{10},$  $\beta_{01},\cdots, \beta_{m,0},\beta_{m-1,1},\cdots, \beta_{1,m-1},\beta_{0,m} \}$ with $\beta_{00}>0$, the \tb{truncated  moment problem} (TMP) entails finding necessary and sufficient conditions for the existence of  a positive Borel measure $\mu$ such that $\supp \mu \sbs\mathbb{R}^2$ and 
\begin{eqnarray*}
\beta_{i,j}\equiv \beta_{( i,j)}=\int x^i y^{j} \,\, d\mu \,\,\,( 0 \le i + j \le 2n;\ i,j \in \mathbb Z_+).
\end{eqnarray*} 
%\begin{eqnarray*}
%\beta_{i,j-i}\equiv \beta_{( i,j-i)}=\int x^i y^{j-i} \,\, d\mu \,\,\,( 0 \le i \le j \le n \in \mathbb Z_+).
%\end{eqnarray*} 
In this context,  we refer to  $\mu$ a \tb{representing measure} (rm) for $\beta$ or the moment matrix $M(n)$ defined below.  
%; if a moment sequence has such a measure, then we say the problem is \ti{soluble}  and the necessary and sufficient conditions  are said to be a solution.
When the order of a moment sequence is even, such as $m = 2n$ for some $n\in {\mathbb N}$, it is possible to define the \tb{moment matrix} $ M(n) \equiv M(n)(\beta^{(2n)})$ of $ \bt$ as follows:
$$
{M}(n)\equiv 
{M}(n)(\beta^{(2n)}):=(\beta_{\textbf{i} +\textbf{j}})_{\textbf{i}, \, \textbf{j}\in \mathbb Z^2_+: |\textbf{i}|,\, |\textbf{j}|\leq n},
$$
{where $\textbf{i}+\textbf{j}$ stands for the coordinate-wise sum and $|\textbf{i}|$ is the sum of coordinates of $\textbf{i}$.}
To guarantee the existence of a representing measure for \(\bt^{(2n)}\), it is essential that the matrix \(M(n)\) is positive semidefinite. However, there are additional conditions that must be met. To examine these conditions, let \(\re[x,y]_{n}\) denote the set of bivariate polynomials in \(\mathbb{R}[x,y]\) with degree at most \(n\). We arrange the columns of \(M(n)\) according to the monomials \(\{\textit{1}, X, Y, X^2, XY, Y^2, \ldots, X^n, \ldots, Y^n\}\) in degree-lexicographic order. For instance, a sextic moment matrix $M(3)$ would appear as follows:
\begin{eqnarray*}
&& \hspace{17mm} \ti{1} \hspace{6mm}   X \hspace{6mm}   Y  \hspace{5mm}  X^2   \hspace{2.5mm}   XY  \hspace{3.5mm}   Y^2 
\hspace{4mm}  X^3   \hspace{2mm}   X^2Y  \hspace{2mm}  X Y^2  \hspace{2.5mm}  Y^3\\
&&\begin{array}{r}
\ \ti{1}\\ X \\ Y\\ X^2\\  XY\\ Y^2 \\ X^3 \\ X^2Y\\ XY^2 \\Y^3
\end{array} \!\!
\left( \begin{array}{c|cc|ccc|cccc}
\beta_{00} & \beta_{10} &\beta_{01} &\beta_{20} &\beta_{11} &\beta_{02} & \beta_{30} & \beta_{21} & \beta_{12} & \beta_{03} \\ \hline
\beta_{10} & \beta_{20} &\beta_{11} &\beta_{30} &\beta_{21} &\beta_{12} & \beta_{40} & \beta_{31} & \beta_{22} & \beta_{13}\\
\beta_{01} & \beta_{11} &\beta_{02} &\beta_{21} &\beta_{12} &\beta_{03} & \beta_{31} & \beta_{22} & \beta_{13} & \beta_{04}\\ \hline
\beta_{20} & \beta_{30} &\beta_{21} &\beta_{40} &\beta_{31} &\beta_{22} & \beta_{50} & \beta_{41} & \beta_{32} & \beta_{23}\\
\beta_{11} & \beta_{21} &\beta_{12} &\beta_{31} &\beta_{22} &\beta_{13} & \beta_{41} & \beta_{32} & \beta_{23} & \beta_{32}\\
\beta_{02} & \beta_{12} &\beta_{03} &\beta_{22} &\beta_{13} &\beta_{04} & \beta_{32} & \beta_{23} & \beta_{32} & \beta_{23}\\ \hline
\beta_{30} & \beta_{40} &\beta_{31} &\beta_{50} &\beta_{41} &\beta_{32} & \beta_{60} & \beta_{51} & \beta_{42} & \beta_{33}\\
\beta_{21} & \beta_{31} &\beta_{22} &\beta_{41} &\beta_{32} &\beta_{23} & \beta_{51} & \beta_{42} & \beta_{33} & \beta_{24}\\
\beta_{12} & \beta_{22} &\beta_{13} &\beta_{32} &\beta_{23} &\beta_{13} & \beta_{42} & \beta_{33} & \beta_{24} & \beta_{15}\\
\beta_{03} & \beta_{13} &\beta_{04} &\beta_{23} &\beta_{13} &\beta_{05} & \beta_{33} & \beta_{24} & \beta_{15} & \beta_{06}\\ 
\end{array} \right).
\end{eqnarray*}

When the matrix $M(n)$ exhibits a column relation,  it can be written as $p(X,Y)={\mathbf 0}$ for some polynomial $p(x,y)=\sum_{ij}a_{ij}x^{i}y^{j}\in \re[x,y]_n$, 
{where each monomial is replaced by the column of $M(n)$ indexed with this monomial and $\mathbf 0$ stands for a zero vector;} this concept is known as functional calculus. 
For convenience, we sometimes refer to this column dependence relation as that of the moment sequence.  
Although column relations themselves are not polynomials, they can be interpreted as such and provide vital information about a representing measure for $M(n)$. Specifically, if $\bt$ has a representing measure, the following condition must be satisfied \cite{CF98}: 
$$ p(X, Y) = \mathbf{0} \implies (p  q)(X, Y) = \mathbf{0} \text{ for every } q \in \re[x,y]\text{ with } \deg(p  q) \leq n.$$
In this case, ${M}(n)$ is said to be \tb{recursively generated} (rg).

Let $\mathcal{Z}(p):=\{(x,y)\in \re^2\colon p(x,y)=0\}$ denote  the zero set of the polynomial $p$. 
{The following result helps determine the support of a representing measure for a truncated moment sequence:

\begin{proposition}\emph{({\cite[Proposition 3.1]{tcmp1}})} \label{r-supp-col} 
Suppose \(\mu\) is a representing measure for \(\beta^{(2n)}\). For \(p \in \re[x,y]_n\),
\[
\operatorname{supp} \, \mu \subseteq \mathcal{Z}(p) \iff p(X,Y) = \mathbf{0}.
\]
\end{proposition}
Now, we aim to define a set contained the support of representing measures;  the \tb{algebraic variety} of $\beta\equiv\beta^{(2n)}$ is given by 
\begin{equation} \label{variety}
\mathcal{V}_\beta \equiv \mathcal{V}\lp {M(n)}  \rp:=\bigcap {}_{p(x,y)=0,\, \deg \, p\, \leq n}\; \mathcal{Z}%
(p). 
\end{equation}% 
%If $\widehat{p}$ denotes the column vector of coefficients of $p$, then we can know $p(X,Y)=\mathcal{M}(n)\widehat{p}$, that is, $p(X,Y)=0$ if and only if $\widehat{p} \in \operatorname{ker}\; \mathcal{M}(n)$.\   
Using Proposition \ref{r-supp-col} the following holds: if $M(n)$ admits a representing measure $\mu$, then 
$$
\operatorname{supp}\; \mu \subseteq \mathcal{V}_\beta  
$$
and
$$  
\operatorname{rank}\; {M}(n)\leq \operatorname{card}\;\operatorname{supp}\;\mu \leq \operatorname{card}\V_\beta.
$$
The second inequality in the above is known as  the \tb{variety condition}.
}

Here, we introduce the \tb{Riesz functional} defined on the space of all monomials. For a given moment sequence \(\beta \equiv \beta^{(m)}\), the linear functional \(\Lambda_\beta\) is given by
$$
\Lambda_\beta \left( \sum_{0 \le i + j \le m} a_{i,j} x^i y^j \right) = \sum_{0 \le i + j \le m} a_{i,j} \beta_{i,j},
$$
where all \(a_{i,j}\) are real numbers and $i,j \in \mathbb Z_+$. 
We say that \(\Lambda_\bt\) is \tb{$K$-positive} for a closed set $K\in \re^2$ if 
\[
\Lambda_\bt(p) \geq 0 \ \  \text{for all } p \in \re[x, y]_m \text{ such that } p|_K \geq 0. 
\]
If, in addition, the conditions \(p|_K \geq 0\) and \(p|_K \not\equiv 0\) imply \(\Lambda_\bt(p) > 0\), then \(\Lambda_\bt\) is said to be \tb{strictly $K$-positive}.  
When \(K = \mathbb{R}^2\), we simply refer to \(\Lambda_\beta\) as \tb{positive} rather than $K$-positive. The $K$-positivity of \(\Lambda_\beta\) is a necessary condition for \(\beta\) to have a \textbf{\(K\)-representing measure}, {i.e., a rm supported on $K$}. Conversely, M. Riesz's classical theorem shows that $K$-positivity is also sufficient to guarantee the existence of \(K\)-representing measures for infinite moment sequences. This result was later extended to \(\mathbb{R}^n\) by E. K. Haviland. Similar results are available for the truncated moment problem, see the reference \cite{tcmp12}.

%\begin{theorem}[Riesz-Haviland's Theorem]\label{thm:RH_full}
%A sequence \(\beta = (\beta_n)_{n \geq 0}\) admits a representing measure supported on the closed set \(K \subset \mathbb{R}\) if and only if \(\Lambda_\beta\) is $K$-positive.
%\end{theorem}

One of the most significant results in truncated moment theory is the Flat Extension Theorem. This theorem states that if \( M(n) \) has a rank-preserving positive extension \( M(n+1) \), then \(\beta^{(2n)}\) possesses a \(\operatorname{rank} M(n)\)-atomic representing measure \cite{tcmp1}. The extension \( M(n+1) \) is referred to as a \tb{flat extension}. A notable special case arises when \(\operatorname{rank} M(n) = \operatorname{rank} M(n-1)\). In this situation, \( M(n) \) is referred to as \tb{flat}, and \(\beta^{(2n)}\) has a unique \(\operatorname{rank} M(n)\)-atomic representing measure.

 We now provide a brief overview of how to find a flat extension of \( M(n) \). Notice that each rectangular block of \( M(n) \) with the same degree moments forms a Hankel matrix. To construct an extension \( M(n+1) \), consider the following form:
\[
M(n+1) = \begin{pmatrix}
M(n) & B \\
B^* & C
\end{pmatrix},
\]
where \( B \) and \( C \) are Hankel matrices with some  new moments of degree \( 2n-1 \) and \( 2n \), respectively.
To ensure that a prospective moment matrix \( M(n+1) \) is positive semidefinite, we need the following classical result:
\begin{theorem}\emph{({\cite{Alb69,Smu}})}
\label{r-smu}
Let \( A \), \( B \), and \( C \) be matrices of complex numbers, with \( A \) and \( C \) being square matrices. Then
\[
\begin{pmatrix}
A & B \\
B^* & C
\end{pmatrix} \geq 0 \iff 
\begin{cases}
A \geq 0, \\
B = AW \text{ for some } W \\
C \geq W^* AW.
\end{cases}
{
\iff
\begin{cases}
A \geq 0, \\
B = AW \text{ for some } W, \\
C \geq B^* A^\dagger B,
\end{cases}
}
\]
where $A^\dagger$ stands for the Moore-Penrose inverse of $A$.
Moreover, 
\[
\operatorname{rank}\begin{pmatrix}
A & B \\
B^* & C
\end{pmatrix} = \operatorname{rank} A \iff C = B^\ast A^\dagger B.
\]
\end{theorem}
\noindent 
%The above result highlights the close relationship with the concept of the Schur complement. % and lets us proceed to the next topic.
Although finding the positive extension may appear straightforward, verifying that the $C$-block is indeed a Hankel matrix is not trivial.

Recently, many intriguing interactions between moment theory and algebraic geometry have been uncovered. Solving truncated moment problems can be interpreted as finding the roots of a system of multivariate polynomial equations. {By Richter's result \cite{Ric57} (see also \cite[Theorem 1.24]{Sch17}), up to recently more often credited to Bayer and Teichmann \cite{BaTe}}, if a moment sequence \(\beta^{(2n)}\) has one or more representing measures, then at least one of these measures must be finitely atomic. Consequently, if a real sequence \(\beta^{(2n)}\) is associated with a finitely atomic representing measure \(\mu\), it can be expressed as
\[
\mu = \sum_{\ell=1}^r \rho_\ell \delta_{(x_\ell , y_\ell )} ,
\]
where \( r \leq \dim \mathbb{R}[x,y]_{2n}\). 
Our task is to determine positive numbers \(\rho_1, \ldots, \rho_r\) (referred to as \emph{densities}) and points \((x_1, y_1), \ldots, (x_r, y_r)\) (referred to as \emph{atoms}) of the measure $\mu$, such that for \(i, j \in \mathbb{Z}_+\) and \(0 \leq i + j \leq 2n\),
\[
\beta_{i,j} = \rho_1 x_1^i y_1^{j} + \cdots + \rho_\ell x_\ell^i y_\ell^{j} . %= \int x^i y^{j} \, d\mu.
%\beta_{i,j-i} = \rho_1 x_1^i y_1^{j-i} + \cdots + \rho_\ell x_\ell^i y_\ell^{j-i} = \int x^i y^{j-i} \, d\mu.
\]

\subsection*{Degree-One Transformation}
The following discusses a method for simplifying the moment problem using {invertible} affine linear transformations (alt), specifically the invariance of moment problems under degree-one transformations. The complex version of this approach is detailed in \cite{tcmp6}, and we adopt the same notation to develop its real counterpart.

For $a,b,c,d,e,f\in \mathbb{R}$ with $bf \neq ce$, define 
$$
\Psi(x,y) \equiv (\Psi_1(x,y), \Psi_2(x,y)) := (a + bx + cy, d + ex + fy)\quad \text{for }
x, y \in \mathbb{R}.
$$
If $\Lambda_\beta$ represents the Riesz functional associated with a given $\beta \equiv \beta^{(2n)}$, then we can construct a new equivalent moment sequence $\tilde{\beta}^{(2n)}$ with $\tilde{\beta}_{i,j} := \Lambda_\beta(\Psi_1^i \Psi_2^{j})$ for \(i, j \in \mathbb{Z}_+\) and \(0 \leq i + j \leq 2n\). It follows that $\Lambda_{\tilde{\beta}}(p) = \Lambda_\beta(p \circ \Psi)$ for all $p \in \re[x,y]_n$. For more details,  refer to \cite{tcmp6}. %[Proposition 1.7]  

%\begin{proposition}\cite[cf. Proposition 1.7]{tcmp6} \label{deg-one}
%\textup{(Invariance under degree-one transformations.)}\  Let $
%{M}(n)$ and $\widetilde{{M}}(n)$ be the moment matrices associated with $\beta^{(2n)}$ and
%$\tilde{\bt}^{(2n)}$, and let $J\hat{p}:=\widehat{p\circ\Psi}$ \textup{(}$p\in\re[x,y]_{n}$\textup{).}\  Then the following are true:

%\begin{enumerate}[(i)]
%\item  \label{lininv(1)}$\tilde{{\mcal M}}(n)=J^{\ast}{\mcal M}(n)J$;

%\item  \label{lininv(2)}$J$ is invertible;

%\item  \label{lininv(3)}$\tilde{\mcal{M}}(n)\geq0\iff {\mcal M}(n)\geq0$;

%\item  \label{lininv(4)}$\rank\tilde{\mcal{M}}(n)=\rank{\mcal M}(n)$;

%\item  \label{lininv(5)}The formula $\mu=\tilde{\mu}\circ\Psi$ establishes a
%one-to-one correspondence between the sets of representing measures for $\bt$ and $\tilde{\bt}$, which preserves measure class and cardinality
%of the support; moreover, $\varphi(\supp\mu )=\supp\tilde{\mu}$;

%\item  \label{lininv(6)}$\mcal{M}\left( n\right) $ admits a flat extension if and
%only if $\tilde{\mcal{M}}\left( n\right) $ admits a flat extension.\  
%\end{enumerate}
%\end{proposition}

\subsection*{Truncated Moment Sequences with an Infinite Algebraic Variety} 
The moment matrix \(M(n)(\beta^{(2n)})\) (or the moment sequence) is said to be \tb{$p$-pure} if its only column relations are those recursively derived from a polynomial \(p \in \mathbb{R}[x, y]_n\). Thus, \(\mathcal{V}_\beta\) is precisely \(\mathcal{Z}(p)\); in other words, the algebraic variety of \(\beta\) is infinite. When the algebraic variety associated with a truncated moment sequence is finite, a clear procedure exists for solving the moment problem \cite{Fia08}. However, concrete solutions for \(M(n)\) with \(n \geq 3\) are scarce and challenging to study {\cite{Fia11,Yoo17i,YoZa,Zal21,Zal22l,Zal22j,Zal24}}. 

\subsection*{Main Results}
{
In this paper we focus on  the TMP on curves of the form 
    $$
    xy=q_mx^m+q(x)+\alpha y,
    \quad \text{where }
    q_m\in \RR\setminus\{0\},\;
    q(x)\in \RR[x]_{m-1}, \;\alpha\in \RR.$$
By applying the alt $(x, y) \mapsto (x + \alpha, q_my)$,  it suffices to solve the TMP on curves with the simpler form
\begin{equation}
    \label{TMP:studied}
    xy =x^m+r(x),\quad
    \text{where }r(x)\in \RR[x]_{m-1}.
    %\text{ and } r(0) \neq 0.
\end{equation}
These curves have a parametrization $(x(t),y(t))=(t,t^{m-1}+\frac{r(t)}{t})$, $t\in\RR$, $t\neq 0$. The TMP on any rational curve is equivalent to a univariate TMP, where some moments are missing and the measure must vanish in certain points. This observation simplifies solving the original TMP, since univariate TMPs are easier to tackle and related technique was already exploited in \cite{YoZa,Zal21,Zal22l,Zal22j,Zal24}. 
To adress the TMP on \eqref{TMP:studied}, the solution to the the strong Hamburger TMP \cite{Zal22j} is required, i.e., the $\RR$--rm of the univariate sequence must vanish in $\{0\}$. 
The original motivation for this paper was the cubic relation of type \eqref{TMP:studied} with $m=3$,
since after applying an alt, every cubic relation has one of four canonical forms, with type \eqref{TMP:studied} and $m=3$ being one of them. The other three types are $y=q(x)$, $y^2=q(x)$ and $xy^2+ay=q(x)$ for some $q\in \RR[x]_{3}$ and $a\in \RR$. A concrete solution to $y=q(x)$ is known \cite{Fia11}, while concrete solutions to the other two types are known  for $y^2=x^3$ \cite{Zal21} and $xy^2=1$ \cite{Zal22j}.

Our first main result (see Theorem \ref{r-main1}) 
is the solution to the TMP 
on \eqref{TMP:studied} in terms of the corresponding univariate TMP with gaps. The solution is a  
Hankel positive semidefinite (psd) completion problem, i.e., the question is when missing anti-diagonals of a partially defined Hankel matrix can be chosen so that the completion is psd and satisties two additional constraints, coming from the solution to the strong Hamburger TMP. Our second main result (see Theorems \ref{r-main3} and \ref{r-main3-singular}) solves this completion problem corresponding to the cubic case ($m=3$ in \eqref{TMP:studied}) in terms of concrete numerical conditions. We also bound the number of atoms in a representing measure with the lowest number of atoms and demonstrate the solution on a numerical example (see Example \ref{eg-1}).
Our third main result (see Theorems \ref{sec:m=4-pure} and \ref{r-main4-non-definite}) solves the completion problem for the quartic case ($m=4$ in \eqref{TMP:studied}) in terms of feasibility of a three-variable system of inequalities. 
}

%%%%%%%%%%%%%%%%%%%%%%%%%%%%%%%%%%%%%%%%%%%%%%%%%%%%%%%%%%%%%%%%%
%%%%%%%%%%%%%%%%%%%%%%%%%%%%%%%%%%%%%%%%%%%%%%%%%%%%%%%%%%%%%%%%%
%%%%%%%%%%%%%%%%%%%%%%%%%%%%%%%%%%%%%%%%%%%%%%%%%%%%%%%%%%%%%%%%%

\section{TMP on \(xy = x^m+ \sum_{s=0}^{m-1} q_s x^s \) with $q_0\neq 0$} 
\label{general-case}

In this section we prove that the $\cZ(p)$--TMP for 
\( p(x, y) = xy -  x^m-\sum_{s=0}^{m-1} q_s x^s \), where each $q_s\in \RR$ and $q_0\neq 0$, is equivalent to the Hankel positive semidefinite completion problem (see Theorem \ref{r-main1}).\\

Let $p$ be as in the first paragraph. By using an alt we may assume that $q_m=1$.
Let  \( \beta \equiv \q \beta_{i,j}^{(2n)} \w \) be a sequence with for $i,j\in \ZZ_+$, $i+ j \le 2n$.
    Let us reorder the indices  $$-2n,-2n+1,\ldots,-1,0,1,\ldots,2(m-1)n-1,2(m-1)n$$
    in the following way:
    \begin{align}
    \label{order-v2}
    \begin{split}
        \text{Row }0&:\quad  0,1,\ldots,2n,\\
        \text{Row }1&:\quad -1,2n+1,\ldots,(2n-1)+m-1,\\
        \text{Row }2&:\quad -2,(2n-1)+m,\ldots,(2n-2)+(m-1)2,\\
        \vdots \phantom{sss} & \\
        \text{Row }k&:\quad -k,(2n-k+1)+(m-1)(k-1)+1,\ldots,(2n-k)+(m-1)k,\\
        \vdots \phantom{sss} &\\
        \text{Row }2n&:\quad -2n,(2n-1)+(m-1)(2n-1)+1,\ldots,(m-1)2n.
    \end{split}
    \end{align}
    Now we adapt Row 1 to Row $2n$, while rewriting Row 0, 
    in the following way
    $$
    \text{Row }k:\quad -k,h(k),h(k)+1,\ldots,(2n-k)+(m-1)k,
    $$
    where 
    \begin{equation}
    \label{def-h}
        h(k):=\max\{2n-k+1+(m-1)(k-1),(m-1)k\}+1
    \end{equation}
    \begin{remark}
    \label{rem:adaptation}
        The reason for this adaptation is the fact that expressing $y$ from the equality $p(x,y)=0$ and then
        raising to the power of the index of the row the relation will be of the form
        $$y^k=\Big(\sum_{s=0}^{m}q_sx^{s-1}\Big)^k=
        q_m^{k}x^{(m-1)k}+\sum_{i=-k}^{(m-1)k} r_i x^i$$
    for some $r_i\in \RR$. Multiplying this relation with $x,\ldots,x^{2n-k}$, we can successively express $x^{(m-1)k+j}$, $j=1,\ldots,2n-k$, out of the equations obtained.
    If $(m-1)k$ is larger than $(2n-k+1)+(m-1)(k-1)$,  then some powers of $x$ will be missing in this procedure.
    The missing powers will be precisely 
    $(2n-k+1)+(m-1)(k-1)+1,\ldots,h(k)$.
   % \hfill $\blacksquare$
    \end{remark}
As explained in Remark \ref{rem:adaptation}, the adaptation may result in the loss of some indices in each row. 
    Let $\cI$ be the set of indices remaining in the sequence after this adaptation.
    We define a map
    $f$ on $\cI$ 
    %the indices in \eqref{indices}
    by the rule
    \begin{align}
    \label{function-f}
    \begin{split}
        f(s)
        &\equiv (f_1(s),f_2(s))\\
        &:=
        \left\{
        \begin{array}{rl}
        \big(s-\#(s)(m-1),\#(s)\big),&    \text{if }s\geq 0,\\[0.2em]
        \big(0, \#(s)\big),& \text{if }s< 0,
        \end{array}
        \right.
    \end{split}
    \end{align}
    where  the index $s$ is contained in  Row $\#(s)$.  

Expressing $y$ from the relation $p(x,y)=0$, 
we see that for $i,j\in \ZZ_+$
we have 
\begin{align}
\label{relation-1}
\begin{split}
 &x^i\lp\sum_{s=0}^{m} q_s x^{s-1}\rp^{j}\\ 
 & \qquad  =x^i \lp  
 \sum_{\substack{k_0 + \ldots + k_{m} = j,\\ k_0,\ldots,k_{m}\in \mathbb Z_+}} \frac{j!}{k_0 ! \cdots k_{m} ! }  
 q_m^{k_0}q_{m-1}^{k_{1}}\cdots q_{0}^{k_{m}}  
 x^{\sum_{s=0}^{m}(m-1-s)k_s}  \rp\\
 & \qquad =  
  \sum_{\substack{k_0 + \ldots + k_m = j,\\ k_0,\ldots,k_m\in \mathbb Z_+ }}  \frac{j!}{k_0 ! \cdots k_m !}
  q_m^{k_0}q_{m-1}^{k_{1}}\cdots q_{0}^{k_{m}}
  x^{\sum_{s=0}^{m}(m-1-s)k_s+i}\\
  & \qquad=
    \sum_{t=-j}^{(m-1)j}q_{j,t}x^{t+i},
  \end{split}
\end{align}
where
\begin{equation*}
\label{coef-q}
    q_{j,t}:=
             \displaystyle\sum_{\substack{k_0 + \ldots + k_m = j,\\ k_0,\ldots,k_m\in \mathbb Z_+,\\
             \sum_{s=0}^{m}(m-1-s)k_s=t
             }}  \frac{j!}{k_0 ! \cdots k_m !}
            q_m^{k_0}q_{m-1}^{k_{1}}\cdots q_{0}^{k_{m}}
\end{equation*}
Note that 
    $$
    q_{j,(m-1)j}=q_m^j=1
    \quad
    \text{and}
    \quad
    q_{j,-j}=q_0^j.
    $$
Hence,
\begin{equation}
\label{largest-term}
x^{(m-1)j+i}
=
x^i\lp\sum_{s=0}^{m} q_s x^{s-1}\rp^{j}
-\sum_{t=-j}^{(m-1)j-1} q_{j,t} x^{t+i}.
\end{equation}
and
\begin{equation}
\label{lowest-term}
x^{-j}
=\frac{1}{q_0^j}
\left[
x^i\lp\sum_{s=0}^{m} q_s x^{s-1}\rp^{j}
-\sum_{t=-j+1}^{(m-1)j} q_{j,t} x^{t+i}
\right].
\end{equation}

    Using the relations \eqref{relation-1}--\eqref{lowest-term} 
    we define a number $\gamma_s$ for every $s\in \cI$
    following the order \eqref{order-v2}
    by
    \begin{equation}
    \label{def-gamma-s}
        \gamma_s:=
        \left
        \{
        \begin{array}{rl}
            \displaystyle \beta_{f(s)}-\sum_{t=-f_2(s)}^{(m-1)f_2(s)-1}
            q_{f_2(s),t}
            \gamma_{t+f_1(s)},&
                \text{if }s\geq 0,\\
            \displaystyle 
            \frac{1}{q_0^{f_2(s)}}
            \Big(
            \beta_{f(s)}- \sum_{t=-f_2(s)+1}^{(m-1)f_2(s)}q_{f_2(s),t}\gamma_{t}\Big),&
                \text{if }s<0.
        \end{array}
        \right.
    \end{equation}
    %where in the case $s<0$ note that $f(s)=(0,s)$.
    Namely, we define $\gamma_s$ in the order
    \begin{align}
    \label{ordering-2}
    \begin{split}
        &
        \underbrace{\gamma_0,\gamma_1,\ldots,\gamma_{2n}}_{j=0},
        \underbrace{\gamma_{-1},\gamma_{h(1)},\gamma_{h(1)+1}\ldots,\gamma_{2n+m}}_{j=1},\ldots,
        \underbrace{\gamma_{-2n+1},\gamma_{h(2n-1)+1}}_{j=2n-1},
        \underbrace{\gamma_{-2n}}_{j=2n}.
    \end{split}
    \end{align}
    If $s$ does not appear in $\cI$, then we call $\gamma_s$ a \textbf{free moment}.
    If $s\in \cI$ and in the definition of $\gamma_s$ there also exist $\gamma_{j}$ which are free moments, 
    %or occur later that $\gamma_{s}$
    %in the ordering \eqref{ordering-2}, 
    then $\gamma_{s}$
    is not uniquely determined and we call it an \textbf{auxiliary moment}.
    If $\gamma_s$ is not free or auxiliary, then it is called a \textbf{fully--determined moment}.

{    
    Let $k\in \NN$.
For 
		$v=(v_0,\ldots,v_{2k} )\in \RR^{2k+1}$
we define the corresponding Hankel matrix as
	\begin{equation*}
		A_{v}:=\left(v_{i+j} \right)_{i,j=0}^k
					=\left(\begin{array}{ccccc} 
							v_0 & v_1 &v_2 & \cdots &v_k\\
							v_1 & v_2 & \iddots & \iddots & v_{k+1}\\
							v_2 & \iddots & \iddots & \iddots & \vdots\\
							\vdots 	& \iddots & \iddots & \iddots & v_{2k-1}\\
							v_k & v_{k+1} & \cdots & v_{2k-1} & v_{2k}
						\end{array}\right).
	\end{equation*}
}
    
    The main result of this section is the following solution to the $\calZ(p)$--TMP for $\beta$.
    
    \begin{theorem}\label{r-main1} 
    Let $ p(x,y)= xy -\sum_{s=0}^m q_sx^s$ with $q_i\in\re$, $q_0\neq 0$ $q_m=1$.
    Given a sequence \( \beta \equiv \q \beta_{i,j}^{(2n)} \w \) for $i,j\in \ZZ_+$, $i+ j \le 2n$,
    %and let $K:=\q (x,y) \in\re^2 : p(x,y)=0 \w$.
%Given a sequence $\beta \equiv \q \beta_{i,j-i}^{(2n)}\w \ (0 \le i \le j \le 2n)$, 
 %$M(n)(\beta)$ is positive semidefinite.
 %and $p$-pure. 
 %,  recursively generated by the column relation $p(X,Y)=\mathbf 0$. 
 let  
    $$\gamma\equiv\gamma^{(-2n,2(m-1)n)}= 
    (\gamma_{-2n}, \gamma_{-2n+1}, \ldots,\gamma_{-1},\gamma_{0},\gamma_{1},\ldots,\gamma_{2(m-1)n})$$
 be defined by the procedure above
 and
 suppose $\mathbf{\gamma_{j_1},\ldots,\gamma_{j_p}}$
 are free moments.
 %from $\beta$. 
Then the following are equivalent: 
\begin{enumerate} 
\item[(i)] $\beta$    admits a $\calZ(p)$--representing measure. 
\item[(ii)] $\gamma$ has  a $\RR$--representing measure $\mu$ 
    for some choice of real values of free moments
    such that $\mu(\{0\})=0$.
\item[(iii)] 
    {There is a
    choice of real values of free moments
    such that 
    $A_{\gamma}$ is positive semidefinite and one of the following holds:
    \begin{enumerate}
        \item $A_{\gamma}$ is positive definite.
        \item 
        $\rank A_{\gamma}
        =\rank A_{\gamma^{(-2n,2(m-1)n-2)}}
        =\rank A_{\gamma^{(-2n+2,2(m-1)n)}}.$
    \end{enumerate}
    }
	\end{enumerate}
\end{theorem}

\begin{proof}
The equivalence $(ii)\Leftrightarrow (iii)$ is \cite[Theorem 3.1]{Zal22j}. It remains to prove the equivalence $(i)\Leftrightarrow (ii)$.
By \cite{Ric57} (see also \cite[Theorem 1.24]{Sch17}), it suffices to prove $(i)\Leftrightarrow (ii)$ for finitely atomic measures, and hence 
 it is enough  to establish the following claim.\\

	\noindent\textbf{Claim.} 
		Let $r\in \NN$.
		A sequence 
			$\gamma$ admits a $r$--atomic $\RR$--rm vanishing in $\{0\}$
		if and only if 
			$\beta$ admits a $r$--atomic $\cZ(p)$--rm.\\

\noindent \textit{Proof of Claim.}  
First we prove the forward implication. 
Let 	
    $\mu_x=\sum_{\ell=1}^r \rho_\ell \delta_{x_\ell}$
be a $\RR$--rm for $\gamma$ where $x_\ell\in \RR\setminus \{0\}$ and $\rho_\ell>0$ for each $\ell$.
We will prove that 
$\mu=\sum_{\ell=1}^r \rho_\ell \delta_{(x_\ell,y_\ell)}$,
where $y_\ell=\sum_{s=0}^{m}q_sx^{s-1}$, is a 
$\calZ(p)$--rm for $\beta$.
We use induction on the index $j$ in $\beta_{i,j}$, where $i+j\leq 2n$:\\

\noindent \textbf{Base of induction:} 
For $j=0$,  we see that
	\begin{align*}
	  \beta_{i,0}
        &=\gamma_i
        =\sum_{\ell=1}^r \rho_\ell x_\ell^{i}
        =\sum_{\ell=1}^r \rho_\ell x_\ell^{i}y_\ell^0,
	\end{align*}
where we used \eqref{def-gamma-s} in the first equality
and 
$\gamma_{i}=\int x^{i} d\mu_x$ in the second.\\

\noindent \textbf{Induction step:}
Assume that the Claim holds for every $j\leq j_0-1$ for some $1\leq j_0\leq 2n$. Let us prove its validity for $j_0$. We consider two cases separately. \\

\noindent
\textbf{Case 1:} \textit{$(i,j_0)$ is in the image of $f$.}

Let $s=f^{-1}((i,j_0))$. Then 
we have 
	\begin{align*}
	  \beta_{i,j_0}
		&=\sum_{t=-j}^{(m-1)j_0}q_{j_0,t}\gamma_{t+i}\\
        &=\sum_{t=-j_0}^{(m-1)j}
            \Big(q_{j_0,t}
            \Big(\sum_{\ell=1}^r \rho_\ell x_\ell^{t+i}\Big)\Big)\\
        &=\sum_{\ell=1}^r \Big(\rho_\ell
            \sum_{t=-j_0}^{(m-1)j_0}
            q_{j_0,t}
            x_\ell^{t+i}\Big)\\
        &=\sum_{\ell=1}^r 
            \Big(\rho_\ell x_\ell^i
            \sum_{t=-j_0}^{(m-1)j_0}q_{j_0,t}
            x_\ell^{t}\Big) 
        \\
	&=\sum_{\ell=1}^r 
            \rho_\ell x_\ell^i y_\ell^{j_0},
	\end{align*}
where we used \eqref{def-gamma-s} in the first equality,
$\gamma_{t+i}=\int x^{t+i} d\mu_x$ in the second,
we interchanged the order of summation in the third,
factored out $x^i_\ell$ from the inner sum in the fourth
and used \eqref{relation-1} for $i=0$ in the fifth.\\

\noindent
\textbf{Case 2:} \textit{$(i,j_0)$ is not in the image of $f$.}

Since $(i,j_0)$ is not in the image of $f$, this means that 
\begin{equation}
\label{not-in-the-image}
    i\neq 0
    \quad \text{and}\quad
    i+(m-1)j_0\leq (2n-j_0+1)+(m-1)(j_0-1).
\end{equation}
Indeed, the first condition in \eqref{not-in-the-image} is clear, since $f(-j)=(0,j)$ for every $0\leq j\leq 2n$, while the second inequality implies that 
$f(i+(m-1)j_0)=(i+m-1,j_0-1)$. If $(i,j_0)$ was in the image of $f$, then 
$f^{-1}((i,j_0))=i+(m-1)j_0$.
The second inequality in \eqref{not-in-the-image} is equivalent to
\begin{equation}
\label{bound-on-i}
    i\leq -m+2n-j_0+2.
\end{equation}
Since the moment sequence must be rg, we must have 
\begin{equation}
\label{rg}
    \beta_{i,j_0}
    =
    \sum_{s=0}^m q_s\beta_{i-1+s,j_0-1}.
\end{equation}
Since $0\leq i-1+s$ for each $s$, $0\leq j_0-1$ and
\begin{align*}
    i+s+j_0-2
    \leq i+m+j_0-2 
    &\leq -m+2n-j_0+2+m+j_0-2=2n,
\end{align*}
where we used $s\leq m$ in the first inequality and
\eqref{bound-on-i} in the second,
it follows that 
each $\beta_{i-1+s,j_0-1}$ in \eqref{rg}
is a part of the original sequence.
We now  see that
\begin{align*}
    \beta_{i,j_0}
    &=
    \sum_{s=0}^m q_s\beta_{i-1+s,j_0-1}\\
   & =
    \sum_{s=0}^m q_s
    \Big(\sum_{\ell=1}^r \rho_\ell x_\ell^{i-1+s}y_\ell^{j_0-1}\Big)\\
    &=
    \sum_{\ell=1}^r
    \rho_\ell x_\ell^iy_\ell^{j_0-1}
    \Big(
    \sum_{s=0}^m \rho_s x_\ell^{s-1}\Big)\\
     &=
    \sum_{\ell=1}^r
    \rho_\ell x_\ell^iy_\ell^{j_0-1}y_\ell \\
    &=
    \sum_{\ell=1}^r
    \rho_\ell x_\ell^iy_\ell^{j_0},  
\end{align*}
where we used \eqref{rg} in the first equality, induction hypothesis in the second, rearranged the double sum in the third and used
\eqref{relation-1} for $i=0$, $j=1$ in the fourth equality.
This concludes the induction step and proves the forward implication.\\

	It remains to prove the backward implication of Claim. 
 Let $\mu=\sum_{\ell=1}^r \rho_\ell \delta_{(x_\ell,y_\ell)}$
 be a $\calZ(p)$--rm for $\gamma$, 
 where $(x_\ell,y_\ell)\in \calZ(p)$ and 
 $\rho_\ell>0$ for each $\ell$.
We will prove that 
$\mu_x=\sum_{\ell=1}^r \rho_\ell \delta_{x_\ell}$
is a rm for $\beta$ which by construction vanishes on $\{0\}$ (since each $x_\ell\neq 0$).
We use induction on the index $i$ in $\gamma_i$ according to the ordering \eqref{order-v2}. For $i=0$,  we have
$\gamma_0=\beta_{0,0}=\sum_{\ell=0}^r \rho_\ell x_\ell^0$
and the statement holds.
Assume now that the statement holds up to some index $s_0$
in \eqref{order-v2} and prove it for $s_1$. We consider two cases separately. \\

\noindent
\textbf{Case 1:} \emph{$\gamma_{s_1}$ is a free moment.}

In this case we are able to  define 
    $\gamma_{s_1}=\sum_{\ell=1}^r\rho_\ell x_\ell^{s_1}$.\\

\noindent
\textbf{Case 2:} \emph{$\gamma_{s_1}$ is not a free moment.}

In this case, $\gamma_{s_1}$ is fully-determined or auxiliary moment, but in both cases $s_1\in \cI$. Let us write
$f(s_1)=(i_1,j_1)$.                                                                                                
If $s_1\geq 0$, then
\begin{align*}
\gamma_{s_1}
&=\beta_{i_1,j_1}-\sum_{t=-j_1}^{(m-1)j_1-1}
            q_{j_1,t}
            \gamma_{t+i_1}\\
&=\sum_{\ell=1}^r\rho_\ell x_\ell^{i_1}y_\ell^{j_1}-
\sum_{t=-j_1}^{(m-1)j_1-1}
            q_{j_1,t}
            \Big(\sum_{\ell=1}^r \rho_\ell x_\ell^{t+i_1}\Big)\\
&=\sum_{\ell=1}^r\rho_\ell x_\ell^{i_1}
\Big(y_\ell^{j_1}-
\sum_{t=-j_1}^{(m-1)j_1-1}
            q_{j_1,t}
            x_\ell^{t}\Big)\\
&=\sum_{\ell=1}^r\rho_\ell x_\ell^{i_1}x_\ell^{(m-1)j_1}
=\sum_{\ell=1}^r\rho_\ell x_\ell^{(m-1)j_1+i_1}
=\sum_{\ell=1}^r\rho_\ell x_\ell^{s_1},
\end{align*}
where we used \eqref{def-gamma-s} in the first equality, induction hypothesis and the definition of free moments in the second,
rearranged the terms in the third, \eqref{relation-1} in the fourth and definition of $(i_1,j_1)$ in the last.

If $s_1<0$, then $i_1=0$
and
\begin{align*}
\gamma_{s_1}
&=\frac{1}{q_0^{j_1}}\Big(\beta_{0,j_1}-\sum_{t=-j_1+1}^{(m-1)j_1}
            q_{j_1,t}
            \gamma_{t}\Big)\\
&=\frac{1}{q_0^{j_1}}\Big[ \sum_{\ell=1}^r\rho_\ell y_\ell^{j_1}-
\sum_{t=-j_1+1}^{(m-1)j_1}
            q_{j_1,t}
            \Big(\sum_{\ell=1}^r \rho_\ell x_\ell^{t}\Big)\Big]\\
&=\sum_{\ell=1}^r\rho_\ell 
\frac{1}{q_0^{j_1}}\Big(y_\ell^{j_1}-
\sum_{t=-j_1+1}^{(m-1)j_1}
            q_{j_1,t}
            x_\ell^{t}\Big)\\
&=\sum_{\ell=1}^r\rho_\ell x_\ell^{-j_1},
\end{align*}
where we used \eqref{def-gamma-s} in the first equality, induction hypothesis and the definition of free moments in the second,
rearranged the terms in the third and \eqref{relation-1} in the last. This proves the backward implication of Claim and concludes the proof of the equivalence $(i)\Leftrightarrow (ii)$ of the theorem. 
\end{proof}

%%%%%%%%%%%%%%%%%%%%%%%%%%%%%%%%%%%%%%%%%%%%%%%%%%%%%%%%%%%%%%%%%
%%%%%%%%%%%%%%%%%%%%%%%%%%%%%%%%%%%%%%%%%%%%%%%%%%%%%%%%%%%%%%%%%
%%%%%%%%%%%%%%%%%%%%%%%%%%%%%%%%%%%%%%%%%%%%%%%%%%%%%%%%%%%%%%%%%

\section{Concrete solution to the TMP on $p(x,y)=xy-x^3-\sum_{i=0}^2 q_ix^i$, $q_0\neq 0$}
\label{sec:m=3}

In this section, we derive explicit numerical conditions for the existence of
free moments in Theorem \ref{r-main1} above for $m=3$, solving the TMP concretely. 
Let us see why $q_0\neq 0$ is given; if not, the TMP would involve a reducible column dependency, which could already be solved using known results. 
We also show that the Carath\'eodory number of the moment sequence of degree $2n$ is $3n$, which represents the minimum number of atoms needed to achieve a representing measure.  Moreover, if a $\calZ(p)$--representing measure exists, then it is $(\rank M(n))$--atomic.
The main results are 
Theorem \ref{r-main3}, which is the solution to $p$-pure cases, 
and Theorem \ref{r-main3-singular}, which solves singular cases.\\

Assume the notation from Section \ref{general-case} and
let $\gamma\equiv \gamma^{(-2n,4n)}$ be defined by \eqref{def-gamma-s}
for $m=3$. We have (see \eqref{def-h})
$$h(k)
=\max\{2n+k-1,2k\}+1
=
\left\{
    \begin{array}{lr}
        2n+k,&  \text{if }k<2n,\\
        4n+1,&  \text{if }k=2n.
    \end{array}
\right.
$$
Hence, $\cI=\{-2n,-2n+1,\ldots,4n-1\}$
and the only free moment is $\mathbf{\gamma_{4n}}$.
We define $\gamma_s$ in the order
\begin{align}
    \label{ordering-m=3}
    \begin{split}
        &
        \underbrace{\gamma_0,\gamma_1,\ldots,\gamma_{2n}}_{j=0},
        \underbrace{\gamma_{-1},\gamma_{2n+1}}_{j=1},\ldots,
        \underbrace{\gamma_{-k},\gamma_{2n+k}}_{j=k},\ldots,
        \underbrace{\gamma_{-2n+1},\gamma_{4n-1}}_{j=2n-1},
        \underbrace{\gamma_{-2n}}_{j=2n}
    \end{split}
\end{align}
by \eqref{def-gamma-s}.
The only auxiliary moment is $\gamma_{-2n}$. Except $\mathbf{\gamma_{4n}}$ and $\gamma_{-2n}$ all the other moments are fully--determined.
Namely, we may rewrite moments for $j=0$, 
\begin{align*}
\gamma_{0}=\beta_{0,0},\ \ \gamma_1=\beta_{1,0}, \ \ \ldots, \ \ 
    \gamma_{2n}=\beta_{2n,0};
\end{align*}
for $j=1$, 
\begin{align*} 
\gamma_{-1}&=\frac{1}{q_0}
    \Big(\beta_{0,1}-\sum_{t=0}^{2} q_{1,t}\gamma_t\Big), \\
\gamma_{2n+1}&=
    \beta_{2n-1,1}-\sum_{t=-1}^{1} q_{1,t}\gamma_{t+2n-1};\\
&\ \vdots\\
\end{align*}
for $j=k$, 
\begin{align*}
\gamma_{-k}&=\frac{1}{q_0^k}
    \Big(\beta_{0,k}-\sum_{t=-k+1}^{2k} q_{k,t}\gamma_t\Big), \\
\gamma_{2n+k}&=
    \beta_{2n-k,k}-\sum_{t=-k}^{2k-1} q_{k,t}\gamma_{t+2n-k};\\
&\ \vdots\\   
\end{align*}
for $j=2n-1$, 
\begin{align*}
\gamma_{-2n+1}&=\frac{1}{q_0^{2n-1}}
    \Big(\beta_{0,2n-1}-\sum_{t=-2n+2}^{4n-2} q_{2n-1,t}\gamma_t\Big), \\
\gamma_{4n-1}&=
    \beta_{1,2n-1}-\sum_{t=-2n+1}^{4n-3} q_{2n-1,t}\gamma_{t};    \end{align*}
for $j=2n$, 
\begin{align*}
\gamma_{-2n}(\mathbf{\gamma_{4n}})&=\frac{1}{q_0^{2n}}
    \Big(\beta_{0,2n}-\sum_{t=-2n+1}^{4n-1} q_{2n,t}\gamma_t-\mathbf{\gamma_{4n}}\Big)=:D-q_0^{-2n}\mathbf{\gamma_{4n}}.
\end{align*}
We introduce a new variable $\mathbf{t}$ for $\mathbf{\gamma_{4n}}$
and write
\begin{small}
\begin{equation*}
    A_{\gamma(\mathbf{t})}
    =
    \kbordermatrix{
                & T^{-n} & T^{-n+1} & \cdots & T^{-1} & \textit{1} & T & \cdots &T^{2n}\\
    T^{-n} & 
        \gamma_{-2n}(\mathbf{t}) & \gamma_{-2n+1} & \cdots &\gamma_{-n+1} & \gamma_{-n} &\gamma_{-n+1} & \cdots & \gamma_{n}\\
    T^{-n+1} & 
        \gamma_{-2n+1} & \gamma_{-2n+2} & \cdots &\gamma_{-n+2} & \gamma_{-n+1} &\gamma_{-n+2} & \cdots & \gamma_{n+1}\\
    \vdots & \vdots & \vdots & &\vdots&\vdots&\vdots&&\vdots\\
    T^{-1} & 
        \gamma_{-n-1} & \gamma_{-n} & \cdots &\gamma_{-2} & \gamma_{-1} &\gamma_{0} & \cdots & \gamma_{2n-1}\\
    \textit{1} & 
        \gamma_{-n} & \gamma_{-n+1} & \cdots &\gamma_{-1} & \gamma_{0} &\gamma_{1} & \cdots & \gamma_{2n}\\
    T & 
        \gamma_{-n+1} & \gamma_{-n+2} & \cdots &\gamma_{0} & \gamma_{1} &\gamma_{2} & \cdots & \gamma_{2n+1}\\
     \vdots & \vdots & \vdots & &\vdots&\vdots&\vdots&&\vdots\\
    T^{2n} & 
        \gamma_{n} & \gamma_{n+1} & \cdots &\gamma_{2n-1} & \gamma_{2n} &\gamma_{2n+1} & \cdots & \mathbf{t}
    }
\end{equation*}
\end{small}
\noindent for the corresponding Hankel matrix.
For $i\leq j$ we write 
    $$\vec{T}^{(i,j)}
    :=
    \begin{pmatrix}
        T^{i} & T^{i+1} & \cdots & T^{j}
    \end{pmatrix}.
    $$
Now the matrix $A_{\gamma(\mathbf{t})}$ has the form
\begin{equation*}
A_{\gamma(\mathbf{t})}
=
\kbordermatrix{
            &T^{-n} & \vec{T}^{(-n+1,2n-1)} & T^{2n}\\
T^{-n} &    D-q_0^{-2n}\mathbf{t} & b^T & \gamma_{n} \\[0.3em]
(\vec{T}^{(-n+1,2n-1)})^T &
            b & A_{\widetilde \gamma} & c \\[0.3em]
T^{2n} &
        \gamma_n & c^T & \mathbf{t}
},
\end{equation*}
where
\begin{align*}
    b^T
    &=\begin{pmatrix}
            \gamma_{-2n+1} & \cdots & \gamma_{-n-1} & \gamma_{-n} & \gamma_{-n+1} & \cdots & \gamma_{2n-1}
        \end{pmatrix},\\
    c^T
    &=\begin{pmatrix}
            \gamma_{n+1} & \cdots & \gamma_{2n-1} & \gamma_{2n} & \gamma_{2n+1} & \cdots & \gamma_{4n-1}
        \end{pmatrix},\\
    \widetilde \gamma
    &=
    (\gamma_{-2n+2},\gamma_{-2n+3},\ldots,\gamma_{4n-2}).
\end{align*}

The following theorem is a solution to the $p$-pure TMP.

\begin{theorem}[Pure case]\label{r-main3} 
Let $ p(x,y)= xy -\sum_{s=0}^3 q_sx^s$ with $q_i\in\re$, $q_0\neq 0$ $q_3=1$.
    Given a $p$-pure sequence \( \beta \equiv \q \beta_{i,j}^{(2n)} \w \) for $i,j\in \ZZ_+$, $i+ j \le 2n$,
    %and let $K:=\q (x,y) \in\re^2 : p(x,y)=0 \w$.
%Given a sequence $\beta \equiv \q \beta_{i,j-i}^{(2n)}\w \ (0 \le i \le j \le 2n)$, 
 %$M(n)(\beta)$ is positive semidefinite.
 %and $p$-pure. 
 %,  recursively generated by the column relation $p(X,Y)=\mathbf 0$. 
 let  
    $$
    \gamma(\mathbf{t})= 
        (\gamma_{-2n}(\mathbf{t}), \gamma_{-2n+1}, \ldots,
        \gamma_{-1},\gamma_{0},\gamma_{1},\ldots,\gamma_{4n-1},\mathbf{t})
    $$
 be defined by the procedure above.
Assume the notation above. 
Let us define 
\begin{align}
\label{r-main-3-def}
\begin{split}
    t_{\min}&:=c^T A_{\widetilde \gamma}^{-1} c, \\
    t_{\max}&:=q_0^{2n}(D-b^T A_{\widetilde \gamma}^{-1} b),\\
    E&:=b^T A^{-1}_{\widetilde \gamma}cc^T A^{-1}_{\widetilde\gamma} b 
    -2b^TA^{-1}_{\widetilde \gamma}c \gamma_n+\gamma_n^2.
\end{split}
\end{align}
Then the following are equivalent: 
\begin{enumerate} 
\item[(i)] $\beta$    admits a representing measure. 
\item[(ii)] $\beta$    admits a ${(3n)}$--atomic representing measure. 
\item[(iii)] 
        $t_{\min} < t_{\max}$ 
    and
    %\begin{equation}
    %    \label{sol:cond}
    $    (t_{\max}-t_{\min})^2
        \geq 
        4q_0^{2n}E.$
    %\end{equation}
\end{enumerate}
\end{theorem}

The following theorem is a solution to the singular $\calZ(p)$--TMP with a finite algebraic variety. 

\begin{theorem}[Singular case]\label{r-main3-singular} 
Let $ p(x,y)= xy -\sum_{s=0}^3 q_sx^s$ with $q_i\in\re$, $q_0\neq 0$ $q_3=1$.
    Given a sequence \( \beta \equiv \q \beta_{i,j}^{(2n)} \w \) for $i,j\in \ZZ_+$, $i+ j \le 2n$,
    with $\rank M(n)<3n$,
 let  
    $$
    \gamma(\mathbf{t})= 
        (\gamma_{-2n}(\mathbf{t}), \gamma_{-2n+1}, \ldots,
        \gamma_{-1},\gamma_{0},\gamma_{1},\ldots,\gamma_{4n-1},\mathbf{t})
    $$
 be defined by the procedure above.
Assume the notation above. 
Then the following are equivalent: 
\begin{enumerate} 
\item[(i)] $\beta$    admits a representing measure. 
\item[(ii)] $\beta$    admits a $(\rank M(n))$--atomic representing measure. 
\item[(iii)] 
    $c=A_{\widetilde\gamma}w$ for some $w\in \RR^{3k-1}$
    and
    $\rank A_{\gamma(t_0)}=\rank A_{\widetilde\gamma}$ for 
    $t_{0}:=c^T A_{\widetilde \gamma}^{+}c$,
    where $A_{\widetilde \gamma}^{+}$ stands for the Moore-Penrose inverse of $A_{\widetilde\gamma}$.
\end{enumerate}
\end{theorem}

\begin{proof}[Proof of Theorem \ref{r-main3}]
Before we prove the equivalences of the theorem, we derive a few claims.
Let us denote by $(A_{\gamma(\mathbf{t})})|_{\vec{T}^{(i,j)}}$ the restriction of $A_{\gamma(\mathbf{t})}$
to a principal submatrix on rows and columns labelled by elements from 
$\vec{T}^{(i,j)}$.
\\

\noindent
\textbf{Claim 1.}
$
A_1:=(A_{\gamma(t)})|_{\vec{T}^{(-n,2n-1)}}\succeq 0
\; \iff\;
t\leq t_{\max}.
$\\

\noindent\textit{Proof of Claim 1.}
We see that 
\begin{equation*}
A_1
=
\kbordermatrix{
            &T^{-n} & \vec{T}^{(-n+1,2n-1)} \\
T^{-n} &    D-q_0^{-2n}t& b^T\\[0.3em]
(\vec{T}^{(-n+1,2n-1)})^T &
            b & A_{\widetilde \gamma}
}.
\end{equation*}
Since $A_{\widetilde \gamma}$ is positive definite, 
Theorem \ref{r-smu} implies the following:
\begin{equation*}
    A_1\succeq 0 
    \quad\iff\quad
    D-q_0^{-2n}t\geq b^T A_{\widetilde \gamma}^{-1} b
    \quad\iff\quad
    t\leq t_{\max},
\end{equation*}
which proves Claim 1. \hfill$\square$\\

\noindent
\textbf{Claim 2.}
$
A_2:=(A_{\gamma(t)})|_{\vec{T}^{(-n+1,2n)}}
\succeq 0
\;
\iff
\;
t_{\min}\leq t.\\
$

\noindent\textit{Proof of Claim 2.}
We see that
\begin{equation*}
A_2
=
\kbordermatrix{
            &\vec{T}^{(-n+1,2n-1)} & T^{2n} \\
(\vec{T}^{(-n+1,2n-1)})^T &  A_{\widetilde\gamma} & c\\[0.3em]
T^{2n} &
            c^T & t
}.
\end{equation*}
Since $A_{\widetilde \gamma}$ is positive definite, Theorem \ref{r-smu} implies the following:
\begin{equation*}
    A_2\succeq 0 
    \quad \iff \quad 
    t \geq c^T A_{\widetilde \gamma}^{-1} c=t_{\min},
\end{equation*}
which proves Claim 2. \hfill$\square$\\

\noindent
\textbf{Claim 3.}
Let 
$\mathbf{t}=t_{\min}+\mathbf{w}$
for $\mathbf{w}>0$
and
\begin{equation}
    \label{def-Q-v2}
    Q(\boldw)
    :=
    -\frac{\boldw^2}{q_0^{2n}}
    +\frac{t_{\max}-t_{\min}}{q_0^{2n}}\boldw
    -
    E
\end{equation}
    be a quadratic polynomial.
Then
\begin{align}
\label{psd-cond-1}
    A_{\gamma(t)}\succeq 0
    \quad&\iff\quad
    %\text{There exists }
    %w\leq t_{\max}-t_{\min}
    %\text{ such that }
    Q(w)\geq 0 \text{ and }w\leq t_{\max}-t_{\min}\\
\label{psd-cond-2}
    \quad&\iff\quad
    (t_{\max}-t_{\min})^2
        \geq 
        4q_0^{2n}E.
\end{align}
\smallskip

\noindent\textit{Proof of Claim 3.}
In particular, for $A_{\gamma(t)} \succeq 0$ we must have $A_1\succeq 0$ and $A_2\succeq 0$. By Claim 1, it follows that $t\leq t_{\max}$
which is equivalent to $w\leq t_{\max}-t_{\min}$.
Since $t>t_{\min}$, we know that $A_2\succ 0$ and then by 
\cite[Formula (0.7.2)]{Zha05} we have 
$$
A_2^{-1}
=
\begin{pmatrix}
    A^{-1}_{\widetilde \gamma} & 0 \\ 
    0 & 0
\end{pmatrix}
+
\frac{1}{w}
\begin{pmatrix}
    A^{-1}_{\widetilde \gamma}cc^tA^{-1}_{\widetilde\gamma} & 
        -A^{-1}_{\widetilde \gamma}c\\ 
    -c^TA^{-1}_{\widetilde \gamma}& 1
\end{pmatrix}.
$$
Using Theorem \ref{r-smu},  we see that 
\begin{align*}
%\label{num-test-3}
\begin{split}
    &A_{\gamma}\succeq 0 \\
    &\iff 
    D-\frac{t_{\min}+{w}}{q_0^{2n}}\geq 
    \begin{pmatrix}
    b^T & \gamma_{n}    
    \end{pmatrix} 
    A_{2}^{-1} 
    \begin{pmatrix}
    b \\ \gamma_{n}    
    \end{pmatrix}     
    \\
    &\iff 
    -\frac{w}{q_0^{2n}}
    +\big(D-b^T A_{\widetilde \gamma}^{-1}b-\frac{t_{\min}}{q_0^{2n}}\big)
    -
    \frac{1}{w}
    (b^T A^{-1}_{\widetilde \gamma}cc^tA^{-1}_{\widetilde\gamma} b 
    -2b^TA^{-1}_{\widetilde \gamma}c \gamma_n+\gamma_n^2
    )\geq 0\\
    &\iff 
    Q(w)
    \geq 0,
\end{split} 
\end{align*}
where the last equivalence follows after multiplying by $w$ (which is positive)
and the definition of $Q$. This proves the equivalence in \eqref{psd-cond-1}.
Since 
$Q(w_0)
=
\frac{(t_{\max}-t_{\min})^2}{4q_0^{-2n}}-E$
is a maximum of $Q$
attained in 
$w_0
=\frac{t_{\max}-t_{\min}}{2}$,
this gives the equivalence
\eqref{psd-cond-2}.
\hfill$\square$\\

Let us now prove $(i)\Rightarrow (iii)$. Since $\beta$ has 
a representing measure, then by Theorem \ref{r-main1} there exists $\gamma_{4n}$
such that $\gamma$ has a representing measure. In particular, this means there is $t\in \RR$, such that $A_{\gamma}\succeq 0$.  
By Claims 1 and 2, this  particularly implies that
$t_{\min}\leq t_{\max}$ holds. Next let us show that the inequality is strict.
Assume on the contrary that $t_{\min}=t_{\max}$. This means that 
$\gamma_{4k}$ must be precisely $t_{\min}=t_{\max}$.
Hence, the first and the last column of $A_{\gamma}$ are in the span of the intermediate ones and $\rank A_\gamma=3n-1$, whence 
$\gamma$ has a $(3n-1)$--atomic representing
measure by \cite[Theorem 3.1]{Zal22j}. But then $\beta$ also admits a $(3n-1)$--atomic representing measure, which is a contradiction with $\rank M(n)=3n$, because $\beta$ is $p$-pure. This proves that $t_{\min}<t_{\max}$ holds.
By Claim 3, 
also the second inequality in $(iii)$ holds.\\

Next we prove $(iii)\Rightarrow (ii)$.
To prove that $\beta$ admits a representing measure containing 
    $\rank M(n)=3n$
atoms, we have to show by \cite[Theorem 3.1]{Zal22j} 
there exists a choice of $t$ such that 
\begin{equation}
\label{rank-cond}
    A_{\gamma(t)}\succeq 0
    \quad\text{and}\quad
    3n=\rank A_{\gamma(t)}=\rank A_1=\rank A_2.
\end{equation}
By Claims 1 and 2 above, 
$t\in [t_{\min},t_{\max}]$. If $t$ is equal to one of $t_{\min}$ or $t_{\max}$,
then \eqref{rank-cond} cannot hold due to singularity of $A_1$ or $A_2$, and so 
$t\in (t_{\min},t_{\max})$. 
By assumption in $(iii)$ and the equivalences in Claim 3, there exists $w\in (0,t_{\max}-t_{\min})$ such that
$Q(w)=0$ with $Q$ as in 
\eqref{def-Q-v2}.
For this $w$ we see that  $A_{\gamma}\succeq 0$ and $\rank A_{\gamma}=3n$.
Since $t:=t_{\min}+w\in (t_{\min},t_{\max})$
also the other two rank conditions in \eqref{rank-cond} hold. This concludes the proof of 
$(iii)\Rightarrow (ii)$.\\

Finally, $(ii)\Rightarrow (i)$ is trivial. 
\end{proof}

\begin{proof}[Proof of Theorem \ref{r-main3-singular}]
Since $\rank M(n)<3n$, there must be another column relation not recursively generated by  
    $XY=X^3+q_2X^2+q_1X+q_0$. 
Each additional relation is of the form
\begin{equation}
\label{additional-relation}
    \sum_{\substack{i,j\in \ZZ_+,\\ i+j\leq n}} \alpha_{i,j}X^iY^j=\mathbf{0},
    \quad
    \alpha_{i,j}\in \RR.
\end{equation}
We distinguish between two cases based on additional relations.\\

\noindent 
\textbf{Case 1:} \textit{There exists an additional relation \eqref{additional-relation} with $\alpha_{0,n}=0$.}

The column $X^iY^j$ of $M(n)$ corresponds to the linear combination 
$\sum_{t=-j}^{2j} q_{j,t} T^{t+i}$ of columns $T^\ell$ of $A_{\gamma(\mathbf{t})}$,
where $q_{j,t}$ are as in \eqref{relation-1} for $m=3$.
If $(i,j)\neq (0,n)$, then the exponent $t+i$ can run only from $-n+1$ to $2n-1$.
So the relation \eqref{additional-relation} in Case 1 gives a relation between the columns of $A_{\widetilde \gamma}$. But then \cite[Theorem 3.1]{Zal22j}
implies that $\gamma(t)$ has a $\RR$--rm vanishing in $\{0\}$ for some $t\in \RR$ 
if and only if 
$\rank A_{\widetilde \gamma}=\rank A_{\gamma(t)}$.
In particular, $\rank (A_{\gamma(t)})|_{\vec T^{(-n+1,2n)}}=\rank A_{\widetilde \gamma}$
and by an analogous proof as for Claim 2 in Thereom \ref{r-main3}, 
$t$ must be equal to $c^TA_{\gamma}^\dagger c$ and $c= A_{\widetilde \gamma}w$ for some $w\in \RR^{3n-1}$.
By Theorem \ref{r-main1} and \cite[Theorem 3.1]{Zal22j}, the equivalences of Theorem \ref{r-main3-singular} in this case follow. 
\\

\noindent 
\textbf{Case 2:} \textit{For every additional relation \eqref{additional-relation},  we have $\alpha_{0,n}\neq 0$.}

Using the relation coming form $Y^{n}=\big(\sum_{i=0}^3 X^3+q_2X^2+q_1X+q_0\big)^n$
and the additional relation \eqref{additional-relation} containing $Y^n$ nontrivially,
we get a nontrivial relation among columns $T^\ell$, $\ell=-n,\ldots,2n-1$ of $A_{\gamma(t)}$. But then by \cite[Theorem 3.1]{Zal22j} for the existence of a 
$\RR$--rm vanishing in $\{0\}$ for $\gamma(t)$, $t\in \RR$, there must be a nontrivial relation among columns $T^\ell$, $\ell=-n+1,\ldots,2n$, containing $T^{2n}$ nontrivially (due to rg). This further implies $\rank (A_{\gamma(t)})|_{\vec T^{(-n+1,2n)}}=\rank A_{\widetilde \gamma}$ and by the same arguments as in Case 1, 
the equivalences of Theorem \ref{r-main3-singular} in this case follow.
\end{proof}

\begin{remark}
Recently, the Carath\'eodory number of real plane cubics with smooth projectivization was studied in \cite{BBS24+}  using tools from algebraic geometry. 
The main results show (see \cite[Section 6]{BBS24+}), that the Carath\'eodory number is at most $3n+1$ for degree $2n$ $p$-pure sequences and characterize in terms of the number of connected components of $\cZ(p)$, when it is $3n$. Note that the cubic curve studied in this section does not satisfy projective smoothness assumption and hence the result about the Carath\'eodory number from Theorem \ref{r-main3} does not follow from \cite{BBS24+}.

    Asymptotic estimates for Carath\'eodory number on affine plane curves have been recently studied also in \cite{DK21} and \cite{RS18}.
\end{remark}

The following example demonstrates the solution to the $\calZ(p)$--TMP for $m=4$.

\begin{example}\label{eg-1}
Consider \(\beta \equiv \beta^{(8)}\) with moments generated by the 14-atomic representing measure \(\mu = \sum_{\ell=1}^{14} \rho_\ell \delta_{(x_\ell, y_\ell)}\), where \(\rho_\ell = \frac{1}{14}\), \(x_\ell = \ell\), and \(y_\ell = \frac{(x_\ell + 1)(x_\ell + 2)(x_\ell + 4)}{x_\ell}\) for \(\ell = 1, \ldots, 14\). The moments are given by
\begin{align*}
&\beta_{00}=1, \     \beta_{10}=\frac{15}{2}, \   \beta_{01}= \frac{88829303}{630630}, \ \ldots, \\
&\beta_{80}= \frac{443370241}{2}, \ \ldots , \\
&\beta_{08}=\frac{2248747733666520927131582212659085688086421341014376774177}{237301654241203443784531432580505468750}.    
\end{align*}

\noindent 
Using \textit{Mathematica}, we find the row-reduced form of the moment matrix \( M{(4)}(\beta) \), which shows that it is both positive semidefinite and \(p\)-pure, where \( p(x, y) = xy - x^3 - 7x^2 - 14x + 8 \). The columns $X^3$, $X^4$, and $X^3 Y$ in $\mcal C_{M(4)}$ are linearly dependent, and so $\rank M(4)=12$. Following the procedure form Section \ref{general-case},
%outlined in the proof of Theorem \ref{r-main1}, 
we obtain the associated univariate strong moment sequence \(\gamma \equiv \gamma^{(-8, 16)}\) as follows: 
\begin{align*}
\gamma_{-8}&= {\scriptstyle \frac{1400837170807195875714994726099439569487325591594631638817-237301654241203443784531432580505468750  \beta_{2,7}}{3981261110361986276316941303192577638400000000}} , \\
\gamma_{-7}&={\scriptstyle\frac{795732381288691429080031515331406184509}{11048010629265141181920694037053440000000}} , \\
\gamma_{-6}&={\scriptstyle\frac{445570839299219762020391212081493}{6131652030894184250150235340800000}},\\
&\ \vdots \\
\gamma_{14}&={\scriptstyle\frac{2405869901763265}{2}} , \\  
\gamma_{15}&={\scriptstyle\frac{32512083310326375}{2}}, &   \\
\gamma_{16}&={\scriptstyle\frac{2596336578534357052143750 \beta_{2,7}-14754296464684589107824850551429749877576317}{2596336578534357052143750}},
\end{align*}
where $\beta_{2,7}$ is a parameter.

\noindent

A calculation shows that no value of \(\beta_{2,7}\) such that 
$$\rank A_{(\gamma_{-8},\ldots, \gamma_{16})} 
= \rank A_{(\gamma_{-8},\ldots,\gamma_{14})}
= \rank A_{(\gamma_{-6},\ldots, \gamma_{16})}.$$ 
Another possibility for having a representing measure is \(A_{\gamma}\succ 0\), which corresponds to the following approximation: 
\begin{align}\label{e-keyb27}
    5.9031917636064208814 \times 10^{18} < \beta_{2,7} < 5.9031917636066715225 \times 10^{18},
\end{align}

In this case, \(\beta\) supports infinitely many 13-atomic representing measures. Alternatively, if \(A_{\gamma}\succeq 0\), \(\rank A_{\gamma} = 12\) and \(A_{\gamma}\) is recursively generated in both directions, this occurs at the endpoints of the inequality in (\ref{e-keyb27}). In particular, if $\beta_{2,7}\approx 5.9031917636064208814 \times 10^{18}$, then the zeros of the generating function are given by 
\begin{align*}
t_1 &\approx  9.35449 \times 10^{-11},  &t_2 &\approx 1 , &t_3 &\approx 2.00001, \\
t_4 &\approx  3.00159, &t_5 &\approx 4.03688, &t_6 &\approx5.23594, \\
t_7 &\approx  6.70231 , &t_8 &\approx    8.37317 , &t_9 &\approx 10.0875, \\
t_{10} &\approx 11.6566, &t_{11}&\approx  12.9415, &t_{12} &\approx 13.9981.    
\end{align*}
Solving the Vandermonde equation in this case, the densities are 
\begin{align*}
\rho_1 &\approx  -0.000301331,  &\rho_2 &\approx 0.0700711 , &\rho_3 &\approx 0.0747554, \\
\rho_4 &\approx  0.0588108, &\rho_5 &\approx0.0738502, &\rho_6 &\approx 0.0903296, \\
\rho_7 &\approx 0.111516 , &\rho_8 &\approx 0.122776 , &\rho_9 &\approx 0.119394, \\
\rho_{10} &\approx 0.102687, &\rho_{11}&\approx  0.0814315, &\rho_{12} &\approx 0.0720834.    
\end{align*}We have demonstrated that \(\beta\) admits a 12-atomic rm $\sum_{\ell=1}^{12} \rho_\ell \delta_{(t_\ell ,  s_\ell )}$, where 
$$s_\ell = \frac{(t_\ell + 1)(t_\ell + 2)(t_\ell + 4)}{t_\ell},$$ which differs from the initial measure \(\mu\),
used to generate $\beta$.
\end{example}

%%%%%%%%%%%%%%%%%%%%%%%%%%%%%%%%%%%%%%%%%%%%%%%%%%%%%%%%%%%%%%%%%%%%%%%%%%%%%%
%%%%%%%%%%%%%%%%%%%%%%%%%%%%%%%%%%%%%%%%%%%%%%%%%%%%%%%%%%%%%%%%%%%%%%%%%%%%%%i

\section{More concrete solutions to the TMP on $p(x,y)=xy-x^4-q_3x^3-q_2x^2-q_1x-q_0$, $q_0\neq 0$}
\label{sec:m=4}

In this section, we derive more concrete numerical conditions for the existence of
free moments in Theorem \ref{r-main1} to solve the TMP for $m=4$. 
%We also show that the Caratheodory number of $\calZ(p)$ is $3n$. Moreover, if a $\calZ(p)$-%-representing measure exists, then it is $(\rank M(n))$--atomic.
The main results are 
Theorem \ref{r-main4}, which characterizes the existence of a positive definite completion of the corresponding Hankel matrix from Section \ref{general-case} in terms of a system of inequalities, 
while Theorem \ref{r-main4-non-definite} solves the $\calZ(p)$--TMP for the cases
without positive definite completion.\\

Assume the notations introduced in Section \ref{general-case} and
let $\gamma\equiv \gamma^{(-2n,6n)}$ be defined by \eqref{def-gamma-s}
for $m=4$. 
We have (see \eqref{def-h})
$$h(k)
=\max\{2n+2(k-1),3k\}+1
=
\left\{
    \begin{array}{rl}
        2n+2k-1,&  \text{if }k\leq 2n-2,\\
        6n-2,&  \text{if }k=2n-1,\\
        6n+1,&  \text{if }k=2n,
    \end{array}
\right.
$$
Hence, $\cI=\{-2n,-2n+1,\ldots,6n-4,6n-2\}$
and the free moments in $\gamma$ are 
$\mathbf{\gamma_{6n-3}}$, $\mathbf{\gamma_{6n-1}}$, $\mathbf{\gamma_{6n}}$.
We define $\gamma_s$ in the order
\begin{align}
    \label{ordering-m=4}
    \begin{split}
        &
        \underbrace{\gamma_0,\gamma_1,\ldots,\gamma_{2n}}_{j=0},
        \underbrace{\gamma_{-1},\gamma_{2n+1},\gamma_{2n+2}}_{j=1},\ldots,
        \underbrace{\gamma_{-k},\gamma_{2n+2k-1},\gamma_{2n+2k}}_{j=k},\ldots\\
        &\hspace{3cm}
        \ldots,
        \underbrace{\gamma_{-2n+2},\gamma_{6n-5},\gamma_{6n-4}}_{j=2n},
        \underbrace{\gamma_{-2n+1},\gamma_{6n-2}}_{j=2n-1},
        \underbrace{\gamma_{-2n}}_{j=2n}
    \end{split}
\end{align}
by \eqref{def-gamma-s}.
The auxiliary moments are $\gamma_{-2n+1},\gamma_{6n-2}$, and $\gamma_{-2n}$. 
Namely, we may rewrite moments for $j=0$, 
\begin{align*}
\gamma_{0}=\beta_{0,0},\ \ \gamma_1=\beta_{1,0},\ \ \ldots,\ \
    \gamma_{2n}=\beta_{2n,0};
\end{align*}
for $j=1$, 
\begin{align*}    
\gamma_{-1}&=\frac{1}{q_0}
    \Big(\beta_{0,1}-\sum_{t=0}^{3} q_{1,t}\gamma_t\Big),\\
\gamma_{2n+1}&=
    \beta_{2n-2,1}-\sum_{t=-1}^{2} q_{1,t}\gamma_{t+2n-2},\\
\gamma_{2n+2}&=
    \beta_{2n-1,1}-\sum_{t=-1}^{2} q_{1,t}\gamma_{t+2n-1};\\
&\ \vdots
\end{align*}
for $j=2n$,  
\begin{align*}    
\gamma_{-k}&=\frac{1}{q_0^k}
    \Big(\beta_{0,k}-\sum_{t=-k+1}^{3k} q_{k,t}\gamma_t\Big),\\ 
\gamma_{2n+2k-1}&=
    \beta_{2n-k-1,k}-\sum_{t=-k}^{3k-1} q_{k,t}\gamma_{t+2n-k-1},\\
\gamma_{2n+2k}&=
    \beta_{2n-k,1}-\sum_{t=-k}^{3k-1} q_{1,t}\gamma_{t+2n-k};
\end{align*}
for $j=2n-1$,  
\begin{align*} 
\gamma_{-2n+1}&=\frac{1}{q_0^{2n-1}}
    \Big(\beta_{0,2n-1}-\sum_{t=-2n+2}^{6n-4} q_{2n-1,t}\gamma_t
    -\mathbf{\gamma_{6n-3}}
    \Big),\\    
\gamma_{6n-2}&=
    \beta_{1,2n-1}-\sum_{t=-2n+1}^{6n-5} q_{2n-1,t}\gamma_{t+1}-
    q_{2n-1,6n-4}\mathbf{\gamma_{6n-3}};\\
\end{align*}
for $j=2n$,  
\begin{align*}   
\gamma_{-2n}&=\frac{1}{q_0^{2n}}
    \Big(\beta_{0,2n}-\sum_{t=-2n+1}^{6n-4} q_{2n,t}\gamma_t
    -q_{2n,6n-4}\mathbf{\gamma_{6n-3}}
    -q_{2n,6n-3}\gamma_{6n-2} -\\
    &\hspace{2cm} q_{2n,6n-2}\mathbf{\gamma_{6n-1}}-\mathbf{\gamma_{6n}}\Big).
\end{align*}
We need to introduce new variables 
    $\mathbf{t}_1$, $\mathbf{t}_2$, $\mathbf{t}_3$ 
for 
    $\mathbf{\gamma_{6n-3}}$, $\mathbf{\gamma_{6n-1}}$,
    $\mathbf{\gamma_{6n}}$,
respectively.
Let $\mathbf{\ut}:=(\mathbf t_1,\mathbf t_2,\mathbf t_3)$.
Then $A_{\gamma(\mathbf{\ut})}$ is of the form

\begin{small}
\begin{equation*}
    \kbordermatrix{
                & T^{-n} & T^{-n+1} & \cdots & 1 & \cdots  &T^{3n-2} & T^{3n-1} &T^{3n}\\
    T^{-n} & 
        \gamma_{-2n}(\mathbf{\ut}) & \gamma_{-2n+1}(\mathbf{t}_1) & \cdots & \gamma_{-n} &\cdots & 
            \gamma_{2n-2} & \gamma_{2n-1} & \gamma_{2n}\\
    T^{-n+1} & 
        \gamma_{-2n+1}(\mathbf{t}_1) & \gamma_{-2n+2} & \cdots & \gamma_{-n+1} &\cdots & \gamma_{2n-1} & \gamma_{2n} & \gamma_{2n+1}\\
    \vdots & \vdots &\vdots&&\vdots&&\vdots&\vdots&\vdots\\
    \textit{1} & 
        \gamma_{-n} & \gamma_{-n+1} & \cdots & \gamma_{0} &\cdots & \gamma_{3n-2} & \gamma_{3n-1} & \gamma_{3n}\\
    \vdots & \vdots &\vdots&&\vdots&&\vdots&\vdots&\vdots\\
    T^{3n-2} & 
        \gamma_{2n-2} & \gamma_{2n-1} & \cdots &\gamma_{3n-2} &
        \cdots & \gamma_{6n-4} & \mathbf{t}_1 & \gamma_{6n-2}(\mathbf{t}_1) \\
    T^{3n-1} & 
        \gamma_{2n-1} & \gamma_{2n} & \cdots &\gamma_{3n-1} & \cdots & \mathbf{t}_1 &\gamma_{6n-2}(\mathbf t_1) & \mathbf{t}_2\\
    T^{3n} & 
        \gamma_{2n} & \gamma_{2n+1} & \cdots &\gamma_{3n} & \cdots & \gamma_{6n-2}(\mathbf{t}_1) & \mathbf{t}_2 & \mathbf{t}_3
    }
\end{equation*}
\end{small}
for the corresponding Hankel matrix,
where
\begin{align*} 
    \gamma_{-2n}(\mathbf{\ut})
    &=:C+D\boldt_1+E \mathbf{t}_2-q_0^{-2n}\mathbf{t}_3,\\
    \gamma_{-2n+1}(\mathbf t_1)
    &=:F-q_0^{-2n+1}\mathbf{t}_1,\\
    \gamma_{6n-2}(\mathbf{t}_1)
    &=:G-H\mathbf{t}_1.
\end{align*}
For $i\leq j$ we write 
    $$\vec{T}^{(i,j)}
   : =
    \begin{pmatrix}
        T^{i} & T^{i+1} & \cdots & T^{j}
    \end{pmatrix}.
    $$
\subsection{Existence of a positive definite completion $A_{\gamma(\ut)}$}
\label{sec:m=4-pure}

In this subsection,  we will characterize the existence of $t_1,t_2,t_3$ such that $A_{\gamma(t_1,t_2,t_3)}$ is positive definite. The latter is a sufficient condition for the existence of a $\calZ(p)$--rm for $\beta$ by \cite[Theorem 3.9]{CF91} and Theorem \ref{r-main1} above.

Assume that 
\begin{align}\label{pure-assumpiton} 
(A_{\gamma(\mathbf{\ut})})|_{\vec T^{(-n+1,3n-2)}} 
\text{ is positive definite.}    
\end{align}
We then focus on the submatrix 
\begin{equation} 
\label{def-F1}
    F_1(\boldt_1):=(A_{\gamma(\mathbf{\ut})})|_{\vec T^{(-n+1,3n-1)}}.
\end{equation}
Note that
$$p(\mathbf{t}_1)
:=\det\big(F_1(\boldt_1)\big)
=c_2\mathbf{t}_1^2+c_1\mathbf{t}_1+c_0
$$
with $c_0,c_1,c_2\in \RR$. 
Assuming that $(A_{\gamma(\mathbf{\ut})})|_{\vec T^{(-n+1,3n-2)}}\succ 0$, it follows that $c_2<0$.
For the existence of a positive definite completion $A_{\gamma(\ut)}$, the first necessary condition is the following:
\begin{equation}
\label{nec-cond-1}
    p(\mathbf{t_1}) \text{ has a real zero}.
\end{equation}
Assume \eqref{nec-cond-1} is satisfied. Let $(t_1)_-, (t_1)_+\in \RR$,
with $(t_1)_-\leq (t_1)_+$, be real zeroes of $p(\boldt_1)$.  
Then 
    $F_1(\boldt_1)$ 
is positive definite on 
the interval
    $((t_1)_-,(t_1)_+)$,
and positive semidefinite but not definite in $( (t_1)_-,(t_1)_+ )$.
The question is, whether there exists a choice of
\begin{equation}
\label{sol-cond1}
    t_1\in ((t_1)_-,(t_1)_+),
\end{equation}
such that there are $t_2,t_3\in \RR$ with
    $A_{\gamma(\ut)}$ being positive definite.

Second, assuming \eqref{sol-cond1} holds, we observe the submatrix
\begin{equation} 
\label{def-F2}
F_2(\mathbf{\ut}):=\big(A_{\gamma(\mathbf{\ut})}\big)|_{\vec T^{(-n+1,3n)}}.
\end{equation}
By Theorem \ref{r-smu}, we see that
\begin{align}
\label{psd-ineq-1}
\begin{split}
    &F_2(\ut)\succeq 0 
    \quad
    \iff
    \quad
    t_3
    \geq
    \begin{pmatrix}
    z_1^T &
    \gamma_{6n-2}(t_1) &
    t_2
    \end{pmatrix}
    \Big(F_1(t_1)\Big)^{-1}
    \begin{pmatrix}
    z_1\\
    \gamma_{6n-2}(t_1)\\
    t_2
    \end{pmatrix}
\end{split}
\end{align}
with 
    $
    z_1:=
    \begin{pmatrix}
        \gamma_{2n+1} & \gamma_{2n+2} & \cdots & \gamma_{6n-4} &
        t_1
    \end{pmatrix}^T.
    $
Writing 
\begin{align*}
    \widetilde\gamma
    &:=
    (\gamma_{-2n+2},\gamma_{-2n+1},\ldots,\gamma_{6n-4}),\\
    c_1
    &:=
    \begin{pmatrix}
        \gamma_{2n} & \gamma_{2n+1} & \cdots & \gamma_{6n-4} & t_1
    \end{pmatrix}^T,\\
    w_1
    &:=\gamma_{6n-2}(t_1)-c_1^T A_{\widetilde \gamma}^{-1}c_1,
\end{align*}
we have 
$$
\Big(F_1(t_1)\Big)^{-1}=
\begin{pmatrix}
    A^{-1}_{\widetilde \gamma} & 0 \\ 
    0 & 0
\end{pmatrix}
+
\frac{1}{w_1}
\begin{pmatrix}
    A^{-1}_{\widetilde \gamma}c_1c_1^tA^{-1}_{\widetilde\gamma} & 
        -A^{-1}_{\widetilde \gamma}c_1\\ 
    -c_1^TA^{-1}_{\widetilde \gamma}& 1
\end{pmatrix}.
$$
Using this in the inequality \eqref{psd-ineq-1},  we know 
that $F_2(\ut)\succ 0$ is equivalent to 
\begin{align}
\label{ineq-3}
\begin{split}t_3
&>
\begin{pmatrix}
    z_1^T & \gamma_{6n-2}(t_1)
\end{pmatrix} 
\left(A_{\widetilde \gamma}^{-1}
+
\frac{1}{w_1} A_{\widetilde \gamma}^{-1} c_1c_1^T 
A_{\widetilde \gamma}^{-1}\right)
\begin{pmatrix}
    z_1 \\ \gamma_{6n-2}(t_1)
\end{pmatrix}\\
&\hspace{2cm}
-
\frac{2}{w_1}
\begin{pmatrix}
z_1^T & \gamma_{6n-2}(t_1)
\end{pmatrix}
A_{\widetilde \gamma}^{-1}c_1
+
\frac{t_2^2}{w_1}.
\end{split}
\end{align}

Finally, assuming that \eqref{sol-cond1} and \eqref{ineq-3} hold, 
we now examine the entire matrix $A_{\gamma(\mathbf{\ut})}$. 
By Thorem \ref{r-smu}, we see that 
\begin{align}
\label{ineq-4}
\begin{split}
    &A_{\gamma(\ut)}\succeq 0
    \quad\iff\quad 
    \gamma_{-2n}(\ut)\geq 
    \begin{pmatrix}
    z_2^T & \gamma_{2n}    
    \end{pmatrix} 
    \left(F_2(\ut)\right)^{-1} 
    \begin{pmatrix}
    z_2 \\ \gamma_{2n} 
    \end{pmatrix},
\end{split} 
\end{align}
where 
$
z_2:=
\begin{pmatrix}
    \gamma_{-2n+1}(t_1) & \gamma_{-2n+2} & \cdots & \gamma_{2n-1}
\end{pmatrix}^T.
$
Writing 
\begin{align*}
B
&:=F_1(t_1),\\
c_2
&:=
\begin{pmatrix}
\gamma_{2n+1} & \cdots & \gamma_{6n-2}(t_1) & t_2
\end{pmatrix}^T,\\
w_2
&:=
t_3-c_2^T(F_1(t_1))^{-1}c_2,
\end{align*}
we see that
$$
(F_2(\ut))^{-1}
=
\begin{pmatrix}
    B^{-1} & 0 \\ 
    0 & 0
\end{pmatrix}
+
\frac{1}{w_2}
\begin{pmatrix}
    B^{-1}c_2c_2^tB^{-1} & 
        -B^{-1}c_2\\ 
    -c_2^TB^{-1}& 1
\end{pmatrix}.
$$
Using this in the inequality \eqref{ineq-4},  it follows  that
$A_{\gamma(\ut)}\succ 0$ is equivalent to
\begin{align}
\label{ineq-5}
\begin{split}
\gamma_{-2n}(\ut)
&>
z_2^T
\left(
B^{-1}
+
\frac{1}{w_2} B^{-1} c_2c_2^T 
B^{-1}
\right)
z_2
%\\
%&\hspace{1cm}
-
\frac{2\gamma_{2n}}{w_2}
z_2^TB^{-1}c_2
+
\frac{\gamma_{2n}^2}{w_2}.
\end{split}    
\end{align}

\smallskip

By Theorem \ref{r-main1},
the arguments above give sufficient conditions to solve the $p$-pure TMP.

\begin{theorem}[Purely pure case]\label{r-main4} 
Let $ p(x,y)= xy -\sum_{s=0}^4 q_sx^s$ with $q_i\in\re$, $q_0\neq 0$ $q_4=1$.
    Let \( \beta \equiv \q \beta_{i,j}^{(2n)} \w \) for $i,j\in \ZZ_+$, $i+ j \le 2n$, be a $p$-pure sequence. Assume the notation above and \eqref{pure-assumpiton} holds.
If
there exists a triple $(t_1,t_2,t_3)\in \RR^3$ such that
\eqref{nec-cond-1}, \eqref{sol-cond1}, \eqref{ineq-3}, \eqref{ineq-5} hold, then
$\beta$ admits a $\calZ(p)$--representing measure.
\end{theorem}

\subsection{Existence of positive semidefinite completion $A_{\gamma(\ut)}$ with a $\RR$--rm vanishing in $\{0\}$}
In this subsection,  we study the existence of a $\RR$--rm for $\gamma(\ut)$ vanishing in $\{0\}$ in case $\beta$
is not $p$-pure or a triple $(t_1,t_2,t_3)\in \RR^3$ satisfying the conditions in Theorem \ref{r-main4} does not exist. The main result is Theorem \ref{r-main4-non-definite} below.\\

We say a column relation in  $A_{\gamma(t)}$ of the form 
\begin{equation}
\label{relation-non-pure}
    \sum_{i=i_1}^{i_2} a_i T^{i}=\mbf{0},
\end{equation}
where $a_i\in \RR$,
$-n\leq i_1<i_2\leq 3n$, $a_{i_1}\neq 0$, $a_{i_2}\neq 0$,
\textbf{propagates through $A_{\gamma(\ut)}$}, if 
\begin{align}
\label{propagation-relations}
\begin{split}
    \sum_{i=i_1}^{i_2} a_i T^{i-j}&=\mbf{0}\quad \text{ for } j=1,\ldots,n-i_1,\\
    \sum_{i=i_1}^{i_2} a_i T^{i+j}&=\mbf{0}\quad \text{ for } j=1,\ldots,3n-i_2,
\end{split}
\end{align}
are also relations of $A_{\gamma(\ut)}$.

\begin{theorem}[Singular case]\label{r-main4-non-definite} 
Let $ p(x,y)= xy -\sum_{s=0}^4 q_sx^s$ with $q_i\in\re$, $q_0\neq 0$ $q_4=1$.
    Let \( \beta \equiv \q \beta_{i,j}^{(2n)} \w \) for $i,j\in \ZZ_+$, $i+ j \le 2n$, be a sequence such that there does not exist a triple $(t_1,t_2,t_3)\in \RR^3$ satisfying the conditions in Theorem \ref{r-main4}. 
    Assume the notation of Section \ref{sec:m=4}, and Subsection \ref{sec:m=4-pure} above. 
    We write $C:=(A_{\gamma(\mathbf{\ut})})|_{\vec T^{(-n+1,3n-2)}}$.
    Then $\beta$ admits a $\calZ(p)$--representing measure if and only if one of the following holds:
    \begin{enumerate}[(i)]
    \item\label{m-4-sing-pt1} 
        $C\succeq 0$, $C\not\succ 0$ and a relation \eqref{relation-non-pure} 
        satisfied in $C$ propagates through $A_{\gamma(\ut)}$ for $\ut\in \RR^3$, 
        which is uniquely determined using \eqref{propagation-relations}.
    \item\label{m-4-sing-pt2}
        $C\succ 0$, 
        \eqref{nec-cond-1} holds,
        and the relation \eqref{relation-non-pure} satisfied in $F_1((t_1)_-)$,
        for $F_1$ defined by \eqref{def-F1},
        propagates through $A_{\gamma(\ut)}$ for $(t_2,t_3)\in \RR^2$, 
        uniquely determined using \eqref{propagation-relations}.
    \item\label{m-4-sing-pt3} 
        $C\succ 0$, 
        \eqref{nec-cond-1} holds,
        and the relation \eqref{relation-non-pure} satisfied in $F_1((t_1)_+)$,
        for $F_1$ defined by \eqref{def-F1},
        propagates through $A_{\gamma(\ut)}$ for $(t_2,t_3)\in \RR^2$, 
        uniquely determined using \eqref{propagation-relations}.
    \item\label{m-4-sing-pt4}
        $C\succ 0$, \eqref{nec-cond-1} holds, $t_1\in ((t_1)_,(t_1)_+)$,
        $t_3$ is equal to the right hand side of \eqref{ineq-3} for some $t_2$,
        and the relation \eqref{relation-non-pure} satisfied in $F_2(\ut)$,
        for $F_2$ defined by \eqref{def-F2},
        propagates through $A_{\gamma(\ut)}$.
    \item\label{m-4-sing-pt5}
        $C\succ 0$, \eqref{nec-cond-1} holds, $t_1\in ((t_1)_,(t_1)_+)$,
        $t_3$ satisfies \eqref{ineq-3} for some $t_2$,
        $\gamma_{-2n}(\ut)$ is equal to the right hand side of \eqref{ineq-5},
        and in the relation \eqref{relation-non-pure}, satisfied in $A_{\gamma(\ut)}$,
        we have 
        $i_1=-n$ and $i_2=3n$.
    \end{enumerate}
\end{theorem}

\begin{proof}
By assumption, there does not exist a triple $(t_1,t_2,t_3)\in \RR^3$ such 
that $A_{\gamma(\ut)}\succ 0$.  
We distinguish a few cases based on the point where non-definiteness occurs.\\

\noindent 
\textbf{Case 1:} 
\textit{$C\succeq 0$ and $C\not\succ 0$.}

In this case, there is a relation of the form \eqref{relation-non-pure},
%\begin{equation}
%\label{relation-non-pure}
%    \sum_{i=i_1}^{i_2} a_i T^{i}=\mbf{0},
%\end{equation}
where $a_i\in \RR$,
$-n+1\leq i_1<i_2\leq 3n-2$, $a_{i_1}\neq 0$, $a_{i_2}\neq 0$,
among the columns of $C$. By the extension principle \cite[Proposition 2.4]{Fia95}, this relation must hold in any positive semidefinite  completion 
$A_{\gamma(\ut)}$.
By \cite[Theorem 3.1]{Zal22j}, the existence of a $\RR$--rm canishing in $\{0\}$ is equivalent to well--definedness of the completion $A_{\gamma(\ut)}$ determined by propagating the relation \eqref{relation-non-pure} in both directions by \eqref{propagation-relations}.
This gives \eqref{m-4-sing-pt1}.\\
%namely 
%\begin{align}
%\label{propagation-relations}
%\begin{split}
%    \sum_{i=i_1}^{i_2} a_i T^{i-j}&=\mbf{0}\quad \text{ for } j=1,\ldots,n-i_1,\\
%    \sum_{i=i_1}^{i_2} a_i T^{i+j}&=\mbf{0}\quad \text{ for } j=1,\ldots,3n-i_2,
%\end{split}
%\end{align}

\noindent 
\textbf{Case 2:} 
\textit{$C\succ 0$; \eqref{nec-cond-1} holds and $(t_1)_-$ is defined as in Subsection \ref{sec:m=4-pure}.}

In this case, there is a relation of the form \eqref{relation-non-pure} among columns of $F_1((t_1)_-)$, where $F_1(\boldt_1)$ is defined by \eqref{def-F1}, with $i_2=3n-1$. As in Case 1 above, well--definedness of the relations \eqref{propagation-relations} characterizes the existence of a $(\RR\setminus \{0\})$--rm for $\gamma(\ut)$.
This shows \eqref{m-4-sing-pt2}.\\

\noindent 
\textbf{Case 3:} 
\textit{$C\succ 0$; \eqref{nec-cond-1} holds and $t_1=(t_1)_+$ with $(t_1)_+$ defined as in Subsection \ref{sec:m=4-pure}.}

This case is analogous to Case 2 and completes \eqref{m-4-sing-pt3}.\\

\noindent 
\textbf{Case 4:} 
\textit{$C\succ 0$; \eqref{nec-cond-1} holds, $t_1\in ((t_1)_- ,(t_1)_+)$
and $t_3$ is equal to the right hand side of \eqref{ineq-3} for some $t_2$.
}

In this case, there is a relation of the form \eqref{relation-non-pure} among columns of $F_2(\ut)$, where $F_2(\underline{\boldt})$ is defined by \eqref{def-F2}, with $i_2=3n$. As in Case 1 above, well--definedness of the relations \eqref{propagation-relations} characterizes the existence of a $(\RR\setminus \{0\})$--rm for $\gamma(\ut)$.
This proves \eqref{m-4-sing-pt4}.\\

\noindent 
\textbf{Case 5:} 
\textit{$C\succ 0$; \eqref{nec-cond-1} holds, $t_1\in ((t_1)_-,(t_1)_+)$
and $t_3$ satisfies \eqref{ineq-3} for some $t_2$ and $\gamma_{-2n}(\ut)$
is equal to the right hand side of \eqref{ineq-5}.
}

In this case, there is a relation of the form \eqref{relation-non-pure} among columns of $A_{\gamma(\ut)}$ with $i_1=-n$;  a $\RR$--rm vanishing in $\{0\}$ for $A_{\gamma(\ut)}$
exists only if $i_2=3n$. Otherwise, the second type relations from \eqref{propagation-relations} would need to hold contradicting to positive definiteness of $F_2(\ut)$.
This verifies \eqref{m-4-sing-pt5}.
\end{proof}

\ti{Acknowledgment.} 
% The author is also deeply grateful to the referee for many suggestions that led to significant improvements in the presentation. 
Example \ref{eg-1} was obtained using calculations with the software tool \ti{Mathematica} \cite{Wol}. 
%The  author was supported by the National Research Foundation of Korea (NRF) grant funded by the Korea government (MSIT)  (2020R1F1A1A01070552). 


\end{document}